\documentclass{amsart}
\usepackage{amssymb, xspace}
\usepackage[all]{xy}
\usepackage{enumerate}
\usepackage{verbatim}

\newif\ifpdf
\ifx\pdfoutput\undefined
\pdffalse % we are not running PDFLaTeX
\else
\pdfoutput=1 % we are running PDFLaTeX
\pdftrue
\fi

\ifpdf
\usepackage[pdftex]{graphicx}
\else
\usepackage{graphicx}
\fi

\theoremstyle{plain}
\newtheorem{theorem}{Theorem}[section]

\newtheorem{proposition}[theorem]{Proposition}
\newtheorem{corollary}[theorem]{Corollary}

\theoremstyle{definition}
\newtheorem{definition}[theorem]{Definition}

\theoremstyle{remark}

\numberwithin{equation}{section}

\begin{document}

\ifpdf
\DeclareGraphicsExtensions{.pdf, .jpg, .tif}
\else
\DeclareGraphicsExtensions{.eps, .jpg}
\fi

\title{Thompson's Group F and Uniformly Finite Homology}
\author{D. Staley}
\address {Department of Mathematics\\
    Rutgers University,
    Piscataway, NJ 08854-8019}
\email{staley@math.rutgers.edu}

%\thanks{Bryant and Mio were partially supported by NSF grants DMS-0071693
%and DMS-9626624.}
%\thanks{Ferry was partially supported by NSF grants DMS-9971296
%and DMS-9626101.}
%\thanks{ Weinberger was partially supported by NSF grant
%DMS-9803633.}

%\keywords{keywords}

%\subjclass{Primary: subject; Secondary: subject}

\date{\today}

\large

\begin{abstract}
This paper demonstrates the uniformly finite homology developed by Block and Weinberger and its relationship to amenable spaces via applications to the Cayley graph of Thompson's Group $F$.  In particular, a certain class of subgraph of $F$ is shown to be non-amenable (in the F{\o}lner sense).  This shows that if $F$ is amenable, these subsets (which include every finitely generated submonoid of the positive monoid of $F$) must necessarily have measure zero.
\end{abstract}

\maketitle

\section{Introduction} \label{S:intro}

Richard Thompson introduced his group $F$ in 1965, and it has since been extensively studied.  $F$ is finitely presented, has exponential growth, and its abelianization is $\mathbb{Z} \times \mathbb{Z}$.  The question as to whether $F$ is amenable was first posed in 1979.  $F$ is, in a sense, ``on the edge of amenability'', as it is not elementary amenable but does not contain a free subgroup on two generators \cite{BS85}.  For several years it was hoped to provide the first finitely-presented counterexample to the Von Neumann conjecture, until Ol'shanskii and Sapir provided a different counterexample in 2000 \cite{OS02}.

F{\o}lner provided a geometric criterion for the amenability of a group in 1955 \cite{FO55}, based on the existence of subsets of the Cayley graph with arbitrarily small ``boundary''.  This criterion holds for semigroups as well (one may find a proof in \cite{NA64}.  This criterion allows one to extend the definition of ``amenable'' to cover arbitrary graphs.  In 1992, Block and Weinberger extended the definition to an even broader class of metric spaces.  They defined the {\em uniformly finite homology groups} $H^{uf}_n(M)$ of a metric space $M$, and proved that a space $M$ was amenable if and only if $H^{uf}_0(M) \neq 0$.  This paper seeks to apply the results of Block and Weinberger to Thompson's Group $F$.

The results of Block and Weinberger center around the existence of ``Ponzi schemes'', which come from the uniformly finite 1-chains.  In this paper we will be looking only at graphs, and we will define a slightly simpler notion of a ``Ponzi flow'' which works for our purposes.

In Section 2, we give a very brief overview of Thompson's Group $F$.  Readers interested in a more in-depth introduction are referred to \cite{jB04} or \cite{CF96}.  We also will define amenability and F{\o}lner's criterion.  In Section 3, we discuss the results of Block and Weinberger and define Ponzi flows, as well as proving certain results about Ponzi flows on subgraphs of Cayley graphs (namely, a subgraph with a Ponzi flow always has measure zero).  In Section 4, we state and prove the main result:

\begin{theorem} \label {T:main}
Let $k,l$ be positive integers, with $l > 0$.  Let $\Gamma_k^l$ be the subgraph of the Cayley Graph of $F$ induced by vertices which can be expressed in the form $wv$, where $w$ is a positive word in the infinite generating set $\{x_0, x_1, x_2, ...\}$ of length $\leq k$ and $v$ is a positive word in $\{x_0, ..., x_l\}$ (of any length).  Then $\Gamma_k^l$ is not amenable.
\end{theorem}

The case $k=1, l=1$ was proved by D. Savchuk in \cite{DS08}.

A corollary of this theorem is that all finitely-generated submonoids of the positive monoid of $F$ are not amenable, and therefore if $F$ is amenable these sets have measure zero.

\section{Thompson's Group F}\label{S:TGF}

Thompson's group $F$ has been studied for several decades.  It can be described as the group of piecewise-linear homeomorphisms of the unit interval, all of whose derivatives are integer powers of 2 and with a finite number of break points which are all dyadic rationals.  It can also be described as the group with the following infinite presentation:
$$
<x_0,x_1,x_2,x_3,... | x_ix_j = x_jx_{i-1} \forall i>j >
$$

From this presentation, we may see $x_i = x_0x_{i-1}x_0^{-1}$ for $i > 1$, thus this group is finitely generated by $\{x_0, x_1\}$.  It turns out this group is finitely presented as well.  However, it is still useful to consider the infinite generating set $\{x_0, x_1, x_2, ...\}$.  We have the following definition:

\begin{definition}

The {\em positive monoid} of $F$ is the submonoid of $F$ consisting of elements which can be expressed as words in $\{x_0, x_1, x_2, ...\}$, without using inverses.

\end{definition}

Any element of $F$ can be expressed as an element of the positive monoid times the inverse of such an element.  (Elements of $F$ have a normal form which is such a product).

In \cite{jB04}, the group $F$ is studied using {\it two-way forest diagrams} (so called because trees can extend in either direction in the forest, though we will only be studying forests with trees extending to the right).  We will make extensive use of these diagrams when studying the positive monoid in section 4.  We describe the two-way forest diagrams of the positive monoid here, referring the reader to \cite{jB04} for the proofs.

\begin{definition}

A {\it binary forest} is a sequence of binary trees, such that all but finitely many of the trees are trivial (i.e., have 1 node):

\end{definition}

\vspace{.2in}
\centerline {
\includegraphics[width=4in]{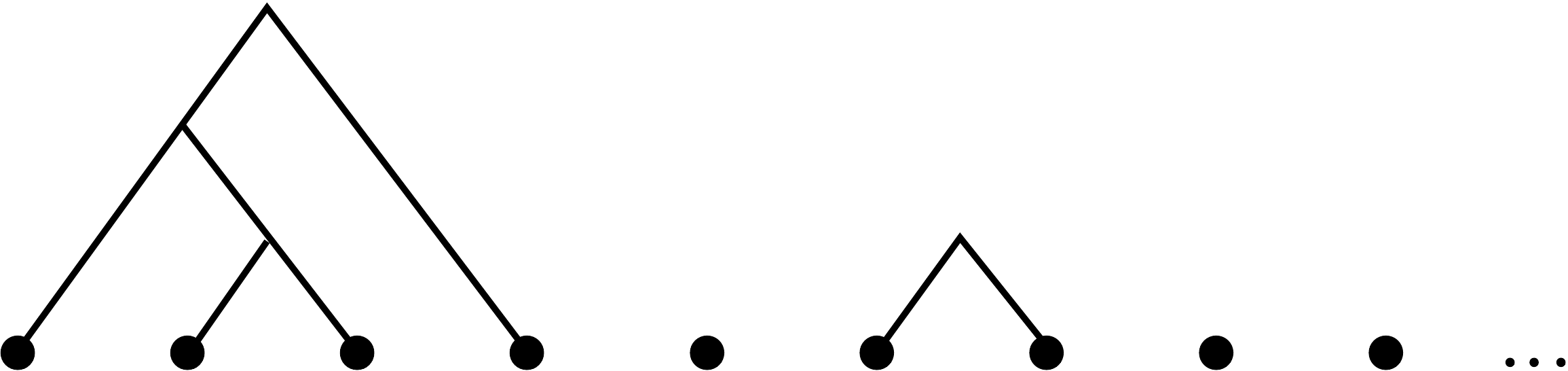}
}

\begin{definition}

 A {\it pointed forest} is a binary forest with a distinguished, or ``pointed'', tree.

\end{definition}

\vspace{.2in}
\centerline {
\includegraphics[width=4in]{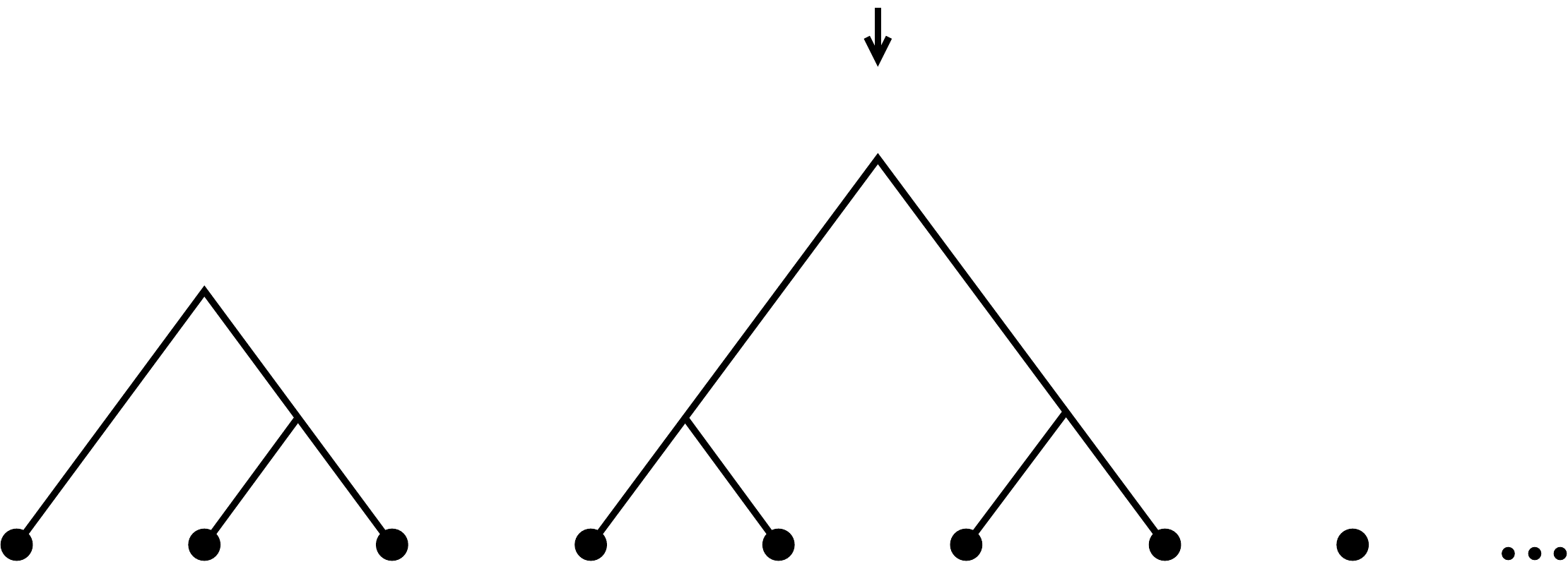}
}

For the remainder of this paper, we will omit the ellipses, and assume a forest diagram or pointed forest diagram has an infinite number of trivial trees continuing to the right.

\smallskip

Each element of the positive monoid of $F$ can be identified with a pointed forest.  The identity element is the pointed forest consisting only of trivial trees, with the the pointer on the leftmost.  Right multiplication by $x_0$ moves the pointer one tree to the right:

\vspace{.1in}
\centerline {
\includegraphics[width=2.9in]{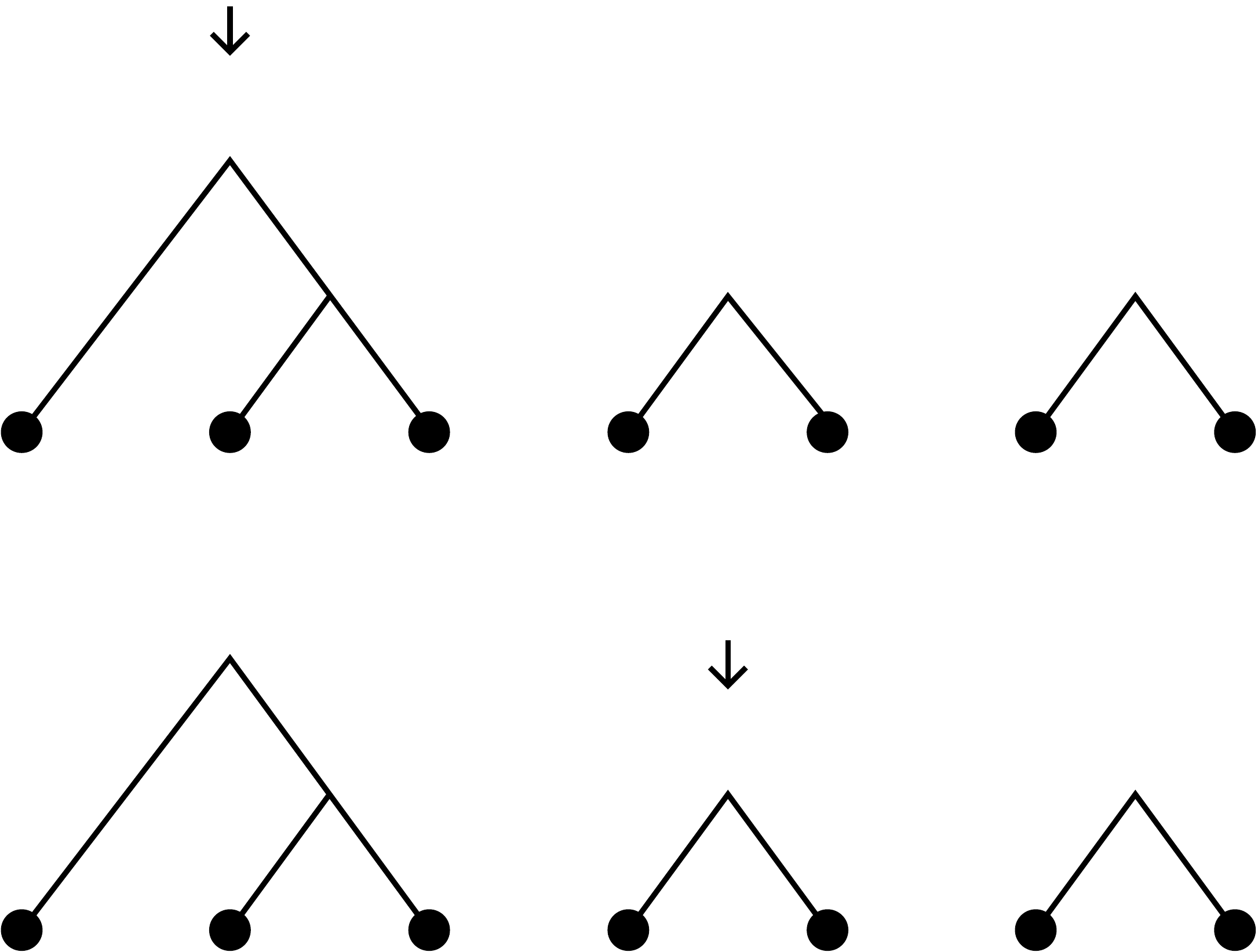}
}
\centerline {Multiplication by $x_0$}
\vspace{.2in}

Right multiplication by $x_1$ adds a caret between the pointed tree and the tree immediately to its right, making a new tree whose left child is the pointed tree and whose right child is the tree to its right.  This new tree becomes the pointed tree:

\vspace{.2in}
\centerline {
\includegraphics[width=3in]{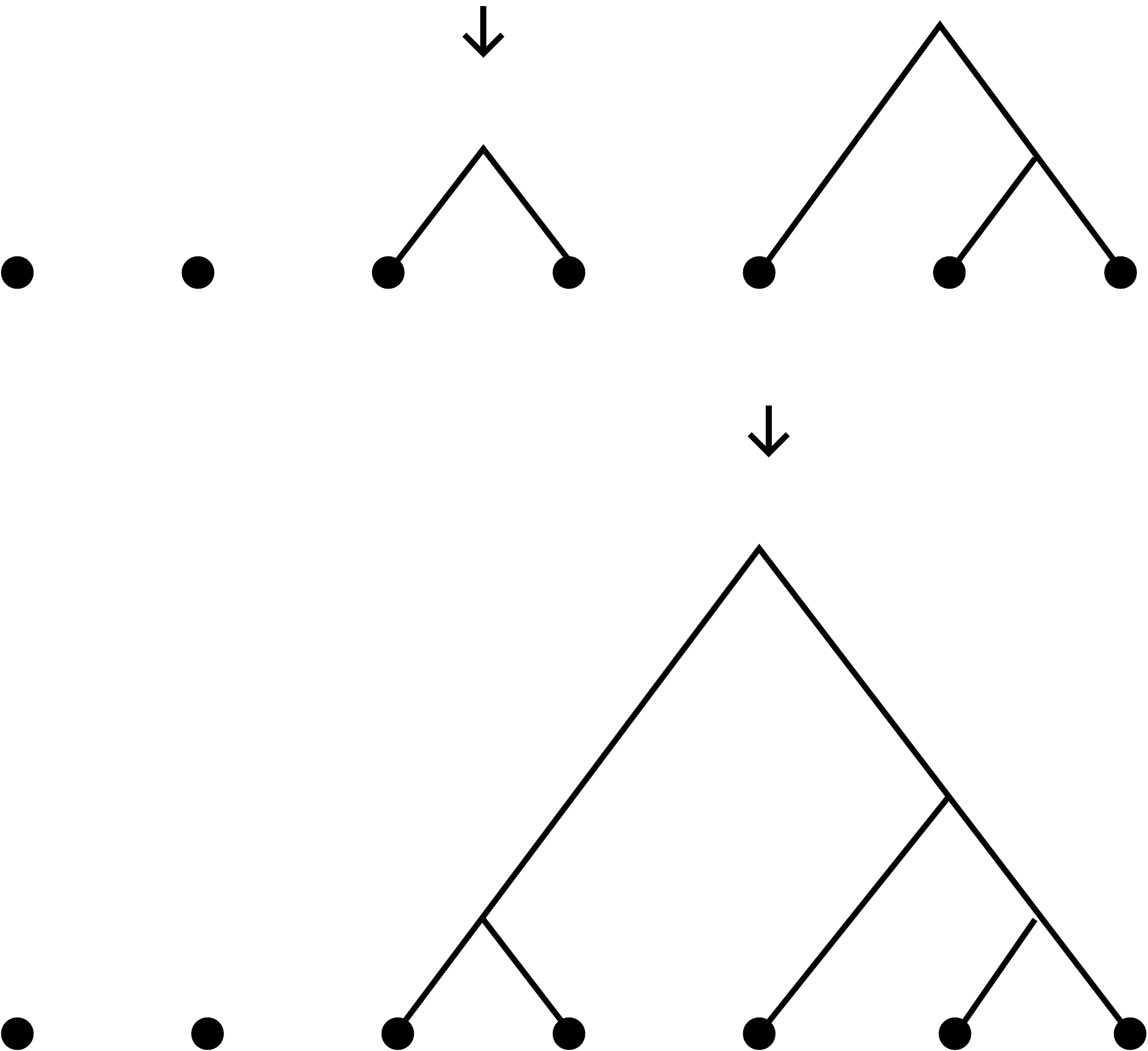}
}
\centerline {Multiplication by $x_1$}
\vspace{.2in}

Since $x_i = x_0^{i-1}x_1x_0^{-(i-1)}$, we can see that right multiplication by $x_i$ moves the pointer $i-1$ trees to the right, adds a caret, and then moves the pointer $i-1$ trees to the left again.  This is equivalent to adding a caret between the trees $i-1$ and $i$ steps away from the pointed tree.

\vspace{.2in}
\centerline {
\includegraphics[width=3in]{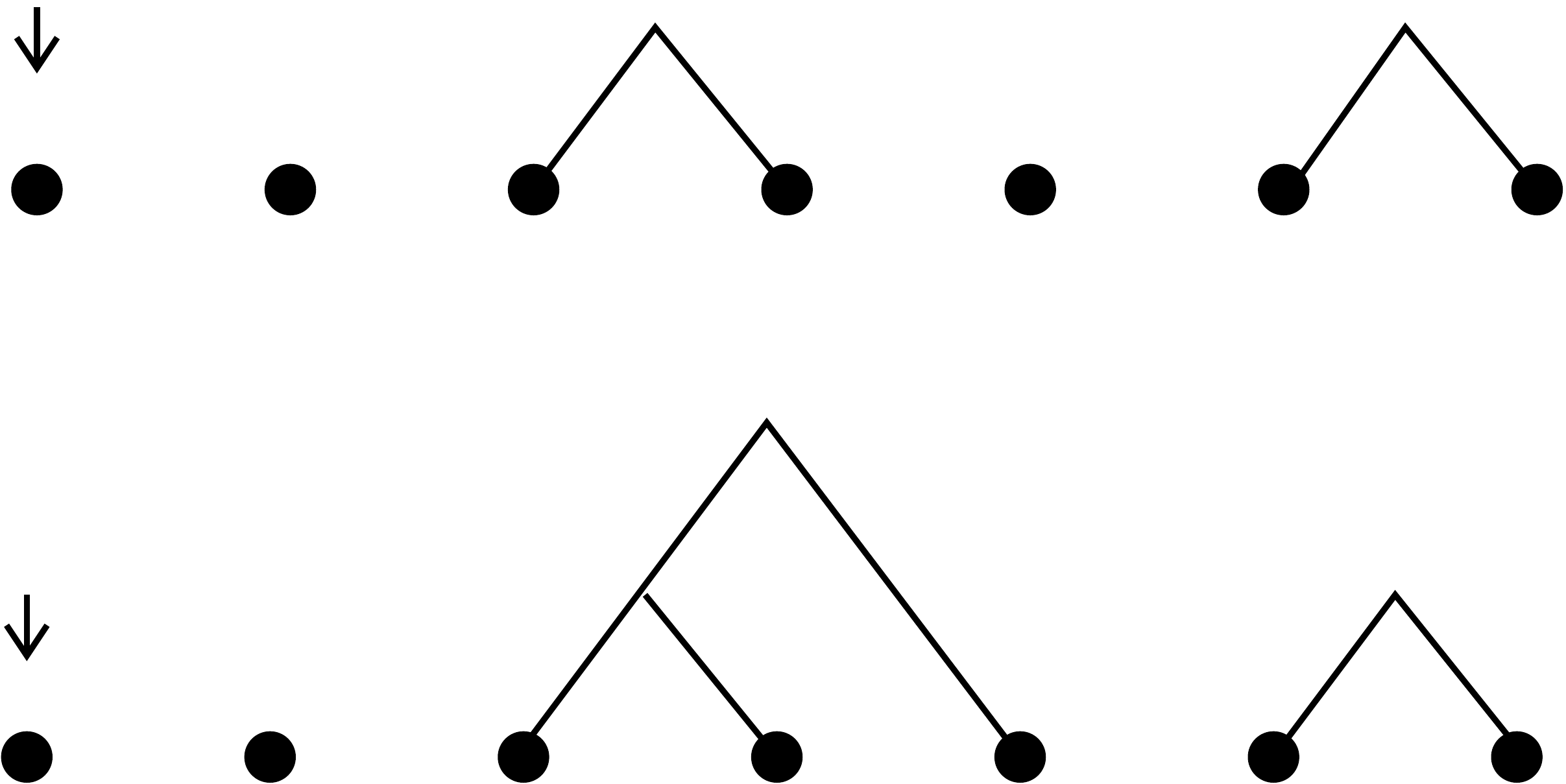}
}
\centerline {Multiplication by $x_3$}
\vspace{.2in}

Multiplication of pointed forests corresponds to ``putting one on top of the other'':  If $P$ and $Q$ are pointed forests, then $PQ$ is the forest obtained by using the trees of $P$ as the leaves of $Q$, with the pointed tree in $P$ acting as the leftmost leaf:

\vspace{.2in}
\centerline {
\includegraphics[width=4in]{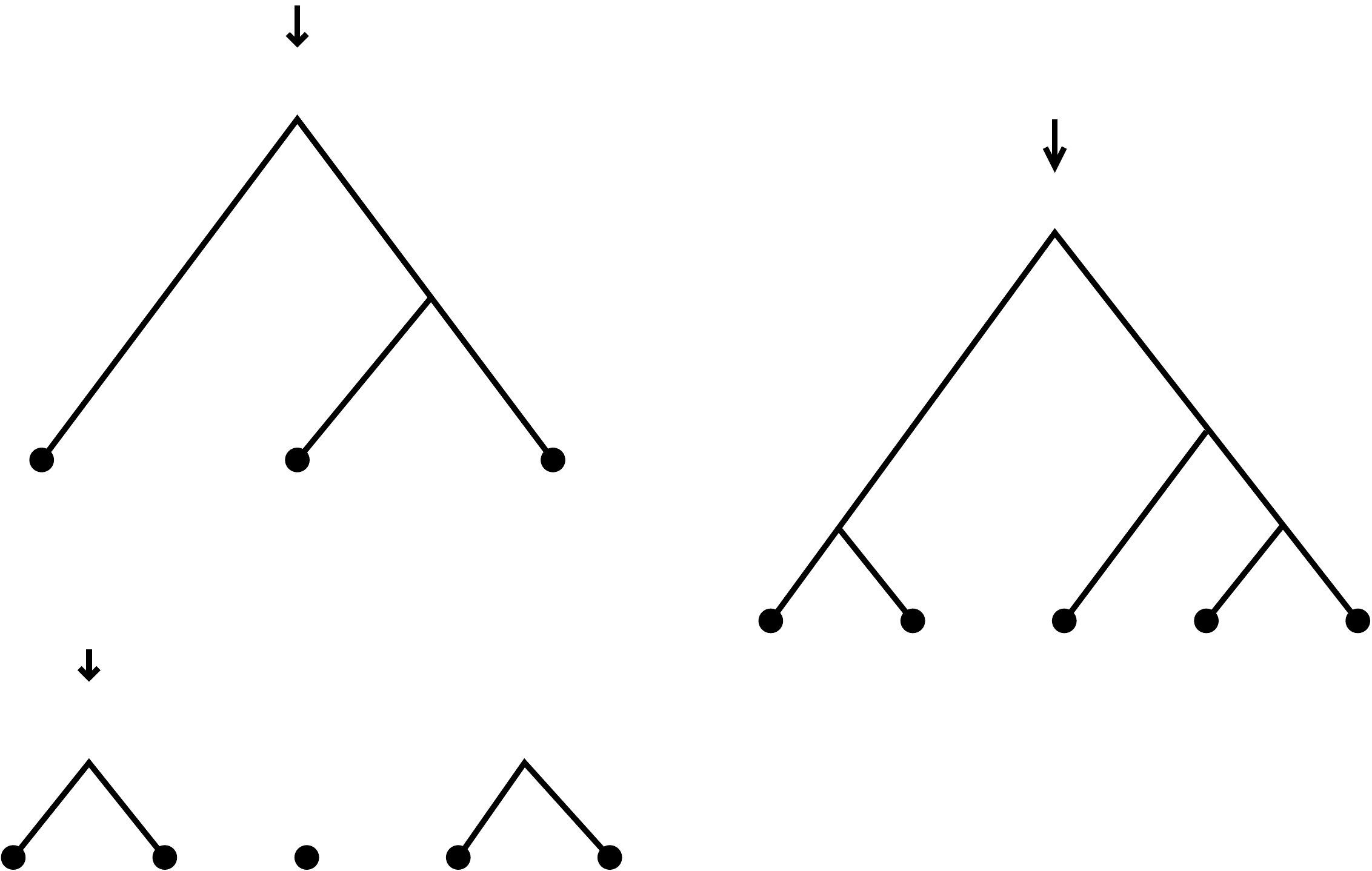}
}
\centerline {Multiplying the two forests on the left yields the forest on the right}
\vspace{.2in}

The pointer is then placed above whatever tree was pointed in $Q$.

It has been a longstanding open question as to whether $F$ is {\em amenable}:

\begin{definition}

A group $G$ is called {\em amenable} if there exists a ``left-invariant measure'' on $G$, that is, a function $\mu : ${\em P}$(G) \rightarrow [0,1]$ such that:

- $\mu(G) = 1$

- $\mu$ is finitely additive:  If $A$ and $B$ are disjoint subsets of $G$, $\mu(A) + \mu(B) = \mu(A\cup B)$.

- $\mu$ is G-invariant:  For any $g \in G$, $A \subset G$, $\mu(A) = \mu(gA)$.

\end{definition}

A useful result for determining amenability is {\it F{\o}lner's Criterion}, which uses the left Cayley graph of $G$ (the graph obtained by taking a generating set $A$ and using $G$ as the vertex set, connecting two vertices $g$ and $g'$ by an edge if $g' = ag$ for some $a \in A$).

\begin{theorem} (F{\o}lner's Criterion):  A group $G$ is amenable if and only if, for any $\epsilon > 0$, there exists a finite subset $A$ of vertices in the Cayley graph of $G$ such that
$$
\frac{\#\partial(A)}{\# A} < \epsilon,
$$
where $\#(A)$ is the number of vertices in $A$, and $\#\partial(A)$ is the number of edges connecting a vertex in $A$ to a vertex outside $A$.
\end{theorem}

Intuitively, F{\o}lner's criterion says there are finite subsets of the Cayley graph whose boundaries are arbitrarily small when compared to the size of the sets themselves.

Since the left- and right- Cayley graphs of a group are isomorphic, F{\o}lner's criterion shows that a left-invariant measure exists on a group if and only if a right-invariant measure exists.  For the remainder of this paper, we will be concerned with {\em right}-invariant measures, and so all Cayley graphs we consider from now on will be right Cayley graphs ($g$ and $g'$ are connected by an edge iff $ga = g'$ for some generator $a$).

F{\o}lner's criterion is very useful, in that it can be applied not only to a group but to any graph.  In particular, we say an arbitrary graph is amenable if F{\o}lner's criterion holds for that graph.  This allows us to state the following proposition:

\begin{proposition}

Let $\Gamma$ be the subgraph of the Cayley graph of Thompson's Group $F$ (using the $x_0, x_1$ generating set) consisting of vertices in the positive monoid and all edges between such vertices.  Then $\Gamma$ is amenable if and only if $F$ is.

\end{proposition}

The proof uses the facts that any finite set in $F$ can be translated into the positive monoid, and that all outgoing edges from the positive monoid are of the $x_0^{-1}$ or $x_1^{-1}$ type.  (A proof can be found in \cite{DS08}).

\section{Uniformly Finite Homology} \label{S:UFH}

This section describes and obtains a few results from the uniformly finite homology of Block and Weinberger.  We will always be considering a graph $\Gamma$ of bounded degree, though many of the results apply to a much broader class of metric spaces.

\begin{definition} \label{D:UF1C}

Let $\Gamma$ be a graph of bounded degree with vertex set $V$.  A {\em uniformly finite 1-chain with integer coefficients} on $\Gamma$ is a formal infinite sum $\sum a_{x,y}(x,y)$, where the $(x,y)$ are ordered pairs of vertices of $\Gamma$, $a_{x,y} \in \mathbb{Z}$, such that the following properties are satisfied:

\bigskip

1) There exists $K>0$ such that $ \forall x,y, |a_{x,y}| < K$

2) There exists some $R>0$ so that $a_{x,y} = 0$ if $d(x,y) > R$

\end{definition}

Notice that condition 2) guarantees that for any fixed $x \in V$, the set of pairs $(x,y)$ such that $a_{x,y} \neq 0$ is finite.  Similarly, the pairs $(y,x)$ with $a_{y,x} \neq 0$ also form a finite set.  This allows us to make the following definition:

\begin{definition} \label{D:PSI}

A uniformly finite 1-chain is a {\em Ponzi scheme} if, for all $x \in \Gamma$, we have $\sum_{v \in \Gamma}a_{v,x} - \sum_{v\in \Gamma}a_{x,v} > 0$.

\end{definition}

We now state the main result of \cite{BW92} which we will use in this paper:

\begin{theorem} \label{T:BW}

Let $\Gamma$ be a graph of bounded degree.  A Ponzi scheme exists on $\Gamma$ if and only if $\Gamma$ is not amenable (in the F{\o}lner sense).

\end{theorem}

We will use a rephrased version of this theorem for the case of our graphs:

\begin{definition} \label{D:PonziFlow}

Let $\Gamma$ be a graph of bounded degree with vertex set $V$. A {\it Ponzi flow on $\Gamma$} will mean a function $G: V\times V\rightarrow \mathbb{Z}$ with the following properties:

\bigskip

i) $G(a,b) = 0$ if there is no edge from $a$ to $b$ in $\Gamma$,

\smallskip

ii) $G(a,b) = -G(b,a)$ for all $a, b \in V$,

\smallskip

iii) $G$ is uniformly bounded, i.e., $\exists N \in \mathbb{Z}$ such that $\forall a, b \in V, G(a,b) \leq N $

\smallskip

iv) For each $a\in V$, $\sum_{b\in G} G(b,a) > 0$.

\end{definition}

Note that the sum in condition ($iv$) is guaranteed to be finite by condition ($i$).  Intuitively, the function $G$ defines a ``flow'' on the graph $\Gamma$, assigning an integer and a direction to each edge, such that for any vertex the total ``inward'' flow is more than the total ``outward'' flow.  This is almost exactly a Ponzi scheme in different language, with the exception that all ``pairs'' must be exactly of distance 1.  However, this difference is unimportant:

\begin{proposition} \label {P:PSPF}

Let $\Gamma$ be a graph of bounded degree.  There exists a Ponzi scheme on $\Gamma$ if and only there exists a Ponzi flow on $\Gamma$.

\end{proposition}

\begin{proof} The ``if'' direction is trivial:  Given a Ponzi flow, we may define our formal sum to be $\sum G(x,y)(x,y)$, and this will be a uniformly finite 1-chain with integer coefficients:  Condition 1) is implied by ($iii$) together with ($iv$), and condition 2) is implied by ($i$).  This 1-chain will be a Ponzi scheme by conditions ($ii$) and ($iv$).

To see the ``only if'' direction, we start with a Ponzi scheme $\sum a_{x,y}(x,y)$, and formally add canceling sums to it to obtain a new scheme $\sum a'_{x,y}(x,y)$ with $a'_{x,y} = 0$ if $d(x,y) > 1$.  We do this in the following way: for each $a_{x,y}$ such that $d(x,y) = n > 1$, let $x = v_0, v_1, ..., v_{n-1}, v_n = y$ be a sequence of vertices forming a path from $x$ to $y$, and add $\sum_{i=0}^{n-1} a_{x,y}(v_i, v_{i+1}) - a_{x,y}(x,y)$ to our sum coefficientwise.  This will still satisfy the inequality of \ref{D:PSI}, since each $v_i$ appears exactly twice in our added sum, once contributing $a_{x,y}$ and once contributing $-a_{x,y}$, leaving the sum unchanged.  The coefficients of $(x,y)$ will cancel, leaving the new coefficient 0.  Thus we have cancelled a coefficient $a_{x,y}$ whose vertices are not adjacent while only affecting other coefficients with adjacent vertices.

Each adjacent pair $(x,y)$ has no more than ${d^{R+1}}\choose{2}$ pairs of vertices within distance $R$, where $d$ is the bound on degree and $R$ is the radial bound from condition 2).  That means that if we choose paths as above for every pair $(x,y)$ with $d(x,y) > 1$ and $a_{x,y} > 0$, each edge will be part of no more than ${d^{R+1}}\choose{2}$ of these paths.   Thus, adding a sum as above for each path will yield us a new well-defined formal sum $\sum a'_{x,y}(x,y)$.  We will have $a_{x,y} = 0$ if $d(x,y) > 1$, since these coefficients have been cancelled.  Furthermore, we have that each $a'_{x,y}$ is bounded by $K ({{d^{R+1}}\choose{2}} + 1)$, thus condition 1) of \ref{D:UF1C} still holds (with a different bound), so we have a uniformly finite 1-chain with $R=1$.  Since we did not change the sums in \ref{D:PSI}, $\sum a'_{x,y}(x,y)$ is still a Ponzi scheme.  Now we simply define $G(x,y) = a_{x,y} - a_{y,x}$, and condition ($i$) is implied by $R=1$, condition ($ii$) is clear, ($iii$) is implied by 1), and ($iv$) is implied by the inequality of \ref{D:PSI}, thus $G$ is a Ponzi flow on $\Gamma$.
\end{proof}

A quantified treatment of Ponzi flows can be found in \cite{LY99}.

If a Ponzi flow exists on a Cayley graph, we then have that there can be no right-invariant measure on the group, since no F{\o}lner sequence exists and this implies the group cannot be amenable.  We give here an elementary proof that existence of a Ponzi flow directly implies no right-invariant measure exists, for an appropriate type of graph on which the notion ``right invariant'' can be defined:

\begin{definition} By a {\em labeled directed graph} we shall mean a directed graph of bounded degree, each of whose edges are labeled by elements from a finite set $\{g_1, g_2, ..., g_n\}$, such that each vertex of $\Gamma$ has at most one incoming edge and one outgoing edge with each label.  It is not necessary that a vertex have an edge with each label.
\end{definition}

The motivating example is the case where $\Gamma$ is a subgraph of a Cayley graph of a group generated by $\{g_1, g_2, ..., g_n\}$, but the results here hold for all labeled directed graphs.

\begin{definition}
Let $\Gamma$ be a labeled directed graph.  Suppose $S$ is a subset of vertices of $\Gamma$.  For $1 \leq i \leq n$, we say $S$ is {\em $g_i$-translatable} if each vertex in $S$ has an outgoing edge labeled by $g_i$.  In such a case we denote by $Sg_i$ the set of vertices with an incoming edge labeled by $g_i$ whose opposite vertex lies in $S$, i.e., the set of vertices at the other ends of the outgoing $g_i$-labeled edges.
\end{definition}

In the case where $\Gamma$ is a subgraph of a Cayley graph, $Sg_i$ is just the right-translate of the elements of $S$ under the group multiplication, and $S$ being $g_i$-translatable simply means $Sg_i \subseteq \Gamma$.  We will abuse notation slightly in the case of one-element sets, so that if $v$ has an outgoing edge labeled by $g_i$, we will call the vertex on the other side of the edge $vg_i$ (which corresponds to standard group multiplication in the case of a Cayley graph).  We will also abuse notation in that we will occasionally identify $\Gamma$ with its vertex set.

If $S$ (or a single vertex $v$) is the $g_i$-translate of some other vertex set, we will call that set $Sg_i^{-1}$ (or the single vertex $vg_i^{-1}$).  Again, this is the same as the standard multiplication in the case of a Cayley graph.

\begin{definition}
Let $\mu$ be a finitely additive measure on the vertex set of $\Gamma$, such that $\mu(\Gamma) = 1$.  Then we say $\mu$ is {\em right-invariant} if for each $g_i$ and each $g_i$-translatable subset $S \subseteq \Gamma$, $\mu(S) = \mu(Sg_i)$.
\end{definition}

If $\Gamma$ is the full Cayley graph of a finitely generated group or semigroup, then this definition of right-invariant measure coincides with the standard one, since every set of vertices is $g_i$-translatable for every $i$.  We now have the terminology to state:

\begin{theorem} \label{T:PonziNoMeasure}
Suppose $\Gamma$ is a labeled directed graph, and has a Ponzi flow $G$.  Then there is no finitely additive, right-invariant measure on $\Gamma$.
\end{theorem}

\begin{proof}
We begin by associating to each vertex $v$ of $\Gamma$ an ordered list of symbols, taken from the alphabet $\{g_1, g_2, ..., g_n\} \cup \{g_1^{-1}, g_2^{-1}, ..., g_n^{-1}\} \cup \{h\}$.  $h$ has no meaning here except as an ``extra'' symbol to be used in the list.  We shall call this list $L_v$, and construct it as follows:  If $v$ has an outgoing edge labeled by $g_1$, and $G(v, vg_1) > 0$, we begin $L_v$ by repeating the symbol $g_1$ $G(v, vg_1)$ times.  If $v$ has an outgoing edge labeled by $g_2$ and $G(v, vg_2) > 0$, we append $G(v,vg_2)$ copies of $g_2$ to the list.  We continue in this fashion, appending $G(v, vg_i)$ copies of $g_i$ to the list if $v$ has an outgoing edge labeled $g_i$ and $G(v, vg_i) > 0$.  For example, if $v$ has an outgoing $g_1$-edge and $g_3$-edge, but no $g_2$-edge, and $G(v, vg_1) = 2$ and $G(v, vg_3) = 4$, then $L_v$ would begin $(g_1, g_1, g_3, g_3, g_3, g_3, ...)$.

Once this is complete for positive powers, we repeat this process with the inverse symbols using incoming edges, i.e., if $v$ has an incoming edge labeled $g_i$ and $G(v, vg_i^{-1}) > 0$, we append $G(v, vg_i^{-1})$ copies of $g_i^{-1}$ to our list, and repeat this process for each $g_i^{-1}$ in turn.  Intuitively, the lists so far are measuring the ``outgoing flow'' at each vertex.

Since $G$ is bounded by some $K$, these lists so far all have length $\leq 2nK$.  We append copies of $h$ to each list so that each list has length $2nK$.  We now have associated to each $v$ an ordered list $L_v$ of length $2nK$ of symbols from the alphabet $\{g_1, ..., g_n, g_1^{-1}, ..., g_n^{-1}, h\}$.

Using these lists, we will now define a family of vertex subsets of $\Gamma$, which we will call $A_L$, for any ordered list $L$ of length $2nK$ of symbols from the alphabet $\{g_1, ..., g_n, g_1^{-1}, ..., g_n^{-1}, h\}$.  $A_L$ will consist of all vertices $v$ for which $L_v = L$.  Note that many of the sets $A_L$ may be empty.  But since each vertex has a unique list associated to it, we have $\Gamma = \bigcup_L A_L$.

Suppose $\Gamma$ has a right-invariant measure $\mu$.  Then since each $A_L$ is disjoint and their union is the entire graph, we have

$$
\sum_{L} \mu(A_L) = 1
$$

Now we define some more subsets of $\Gamma$, which we call $B_g^j$, for $g$ any letter in the alphabet $\{g_1, ..., g_n, g_1^{-1}, ..., g_n^{-1}, h\}$ and $1 \leq j \leq K$.  $B_g^j$ will be the set of vertices $v$ such that $L_v$ contains $j$ or more copies of $g$.  The $B_g^j$ are certainly not disjoint, however, each $B_g^j$ can be expressed as a disjoint union of some of the $A_L$.  Namely, $B_g^j = \bigcup_{L'} A_{L'}$, where the union is taken over all lists $L'$ that contain $j$ or more occurrences of $g$.  Since we are looking at lists of a fixed length with symbols taken from a finite alphabet, there are finitely many such lists and thus the union is finite.  We therefore have

$$
\sum_{g,j}\mu(B_g^j) = \sum_{g,j} \sum_{L'} \mu(A_{L'})
$$

We claim each list $L$ appears exactly $2nK$ times in the double sum on the right-hand side of the above equation.

To prove this claim, we observe that each entry of $L$ causes $A_L$ to be contained in exactly one of the $B_g^j$.  Namely, the $j^{th}$ occurrence of $g$ ensures that $A_L \subset B_g^j$.  For example, if $L$ starts as $(g_1, g_3, g_3, ...)$, then the $g_1$ term guarantees $A_L \subset B_{g_1}^1$, the first $g_3$ term guarantees $A_L \subset B_{g_3}^1$, and the second $g_3$ term guarantees $A_L \subset B_{g_3}^2$. These are the only $B_g^j$ that will contain $A_L$, by the definition of the $B_g^j$.  Thus, since the lists are of length $2nK$, each $A_L$ appears in an $L'$ sum for exactly $2nK$ of the $B_g^j$, and thus appears a total of $2nK$ times in the above sum, proving the claim.

\bigskip

This allows us to explicitly calculate the sum of the measures of the $B_g^j$:

$$
\sum_{g,j}\mu(B_g^j) = 2nK \sum_{L} \mu(A_L) = 2nK
$$

Now, for $g \neq h$, let $C_g^j = B_g^jg$, i.e., $C_{g_i}^j$ is the set of all vertices with an {\em incoming} edge labeled $g_i$ such that $G(vg_i^{-1}, v) \geq j$, and $C_{g_i^{-1}}^j$ is the set of all vertices $v$ with an {\em outgoing} edge labeled $g_i$ such that $G(vg_i, v) \geq j$.  We define $C_h^j = B_h^j$.  Since each $B_g^j$ is $g$-translatable and its translate is $C_g^j$ for $g \neq h$, we have

$$
\sum_{g,j}\mu(C_g^j) = \sum_{g,j}\mu(B_g^j) = 2nK
$$

Now we construct lists $L_v'$ analogously to the $L_v$, with two major changes:  Firstly, we count ``incoming flow'' instead of ``outgoing flow'': we use incoming edges and $G(vg_i^{-1}, v)$ values with the symbols $g_i$ instead of outgoing edges and $G(v, vg_i)$ values, and similarly we use outgoing edges and $G(vg_i, v)$ values with the $g_i^{-1}$ symbols.  Secondly, we append the $h$'s so that each $L_v'$ has length $2nK+1$, not length $2nK$.  Now as before we define $A_L'$ as the set of vertices $v$ such that $L_v' = L$.  We now have for $g \neq h$, $C_g^j = \bigcup_{L'}A_{L'}'$, where the $L'$ are taken over the set of lists with at least $j$ occurrences of $g$.

\smallskip

However, we cannot say the same about $C_h^j$, since it corresponds to $h$'s in the lists $L_v$, not in the $L_v'$.  But since $G$ is a Ponzi flow, the total  ``incoming flow'' at any vertex $v$ is greater than the total ``outgoing flow'', which means that there are more non-$h$ terms in $L_v'$ than there are in $L_v$ for each $v$ (See Definition \ref{D:PonziFlow}).  This means that $L_v'$ has no more $h$ terms than $L_v$ does.  Thus, if we let $C'^j$ be the set of vertices $v$ such that $L_v'$ contains at least $j$ occurrences of $h$, then $C'^j \subset C_h^j$, and so $\mu(C'^j) \leq \mu(C_h^j)$.  But $C'^j = \bigcup_{L'}A_{L'}'$, where the $L'$ are taken over lists with at least $j$ occurrences of $h$.  Thus we may use the exact same arguments as above to obtain

$$
\sum_{g,j}\mu(C_g^j) \geq \sum_{g,j} \sum_{L'} \mu(A_{L'}'),
$$

where the $L'$ sums are taken over lists of length $2nK + 1$ which have at least $j$ occurrences of $g$.  The same argument as the above claim, however, shows that each $A_{L}'$ occurs precisely $2nK+1$ times in the above sum.  Since $\Gamma$ is a disjoint union of all the $A_L'$, this gives us

$$
\sum_{g,j}\mu(C_g^j) \geq (2nK+1)\sum_L \mu(A_L') = 2nK+1
$$.

Since we previously concluded that $\sum_{g,j}\mu(C_g^j) = 2nK$, we have a contradiction. Thus, no right-invariant measure $\mu$ can exist on $\Gamma$.

\end{proof}

\begin{corollary} \label{C:MeasureZero}
If $\Gamma$ is amenable but contains a nonamenable subgraph $P$, then for any right-invariant measure $\mu$ on $\Gamma$, $\mu(P) = 0$.
\end{corollary}

\begin{proof}
If $\mu(P) > 0$, then we can define a measure $\mu'$ on $P$ by setting $\mu'(A) = \frac{\mu(A)}{\mu(P)}$ for $A \subset P$.  Since $\mu(P)$ is constant, $\mu'$ will inherit the properties of right invariance and finite additivity from $\mu$, and $\mu'(P) = \frac{\mu(P)}{\mu(P)} = 1$.  But since $P$ is nonamenable it has a Ponzi flow, and \ref{T:PonziNoMeasure} says no such $\mu'$ can exist, yielding a contradiction.
\end{proof}

\section {Large Nonamenable Subgraphs of $F$} \label{S:LNS}

In this section we will prove the main theorem.

We begin by characterizing the two-way forest diagrams of $\Gamma_k^l$.

Given any binary tree on $n$ nodes, we may remove the top caret, giving us a 2-tree forest on $n$ nodes.  If the left of these two trees is nontrivial, we may remove its top caret to obtain a 3-tree forest.  Suppose we continue to remove the top caret of the leftmost tree, until the leftmost tree is trivial.  Applying this process to any tree gives a function $s$ from the set of binary trees on $n$ nodes to the set of binary forests on $n$ nodes with trivial leftmost tree ($s$ is in fact a bijection).  Note that the inverse of $s$ is defined by starting with a forest on $n$ nodes with trivial leftmost tree, adding a caret between the two leftmost trees in the forest, and repeating this process until the entire forest has been combined into a single tree.

\vspace{.18in}
\centerline {
\includegraphics[width=5in]{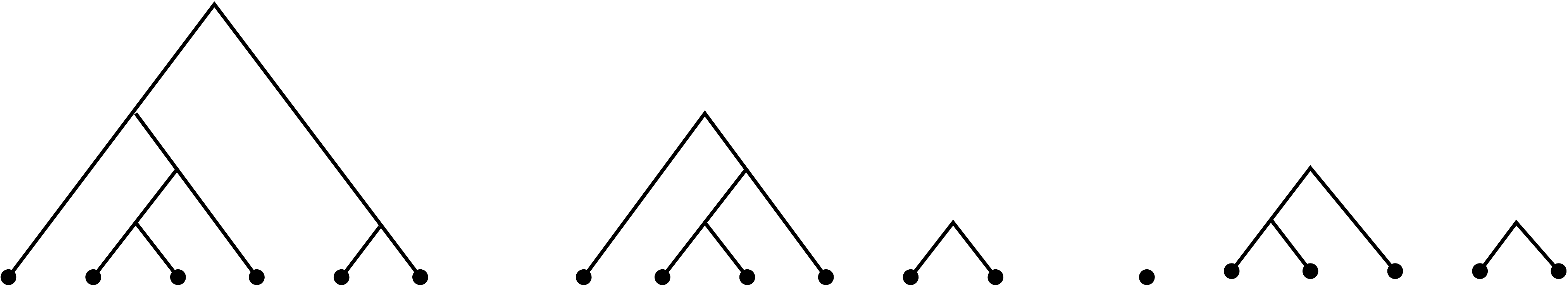}
}
\centerline{Applying $s$ to a tree}
\vspace{.2in}

We extend the definition of $s$ to apply to forests as well as single trees, by applying $s$ separately to each nontrivial tree in the forest. We will define the {\em complexity} of a tree or forest on $n$ nodes to be the minimum number of applications of $s$ required to turn it into into a forest of $n$ trivial trees.

Note that applying $s$ to a tree $T$ leaves a forest whose rightmost tree is the right child of $T$, and the remainder of the forest is $s$ applied to the left child of $T$.  This gives the following:

\begin{proposition} \label{P:calcCpx}

The complexity of a tree is the maximum of the complexity of its left child and one more than the complexity of its right child.

\end{proposition}

We record here a basic property of $s$:

\begin{proposition} \label{P:skills}

Let $T$ be a tree on $n$ nodes, and let $R$ be an $n$-tree forest.  Denote by $RT$ the tree obtained by attaching the roots of $R$ to the nodes of $T$.  If $T$ has complexity $j$, then $s^j(RT)$ consists only of carets in $R$, i.e., every caret from $T$ is removed by $s^j$.

\end{proposition}

\vspace{.2in}
\centerline {
\includegraphics[width=3in]{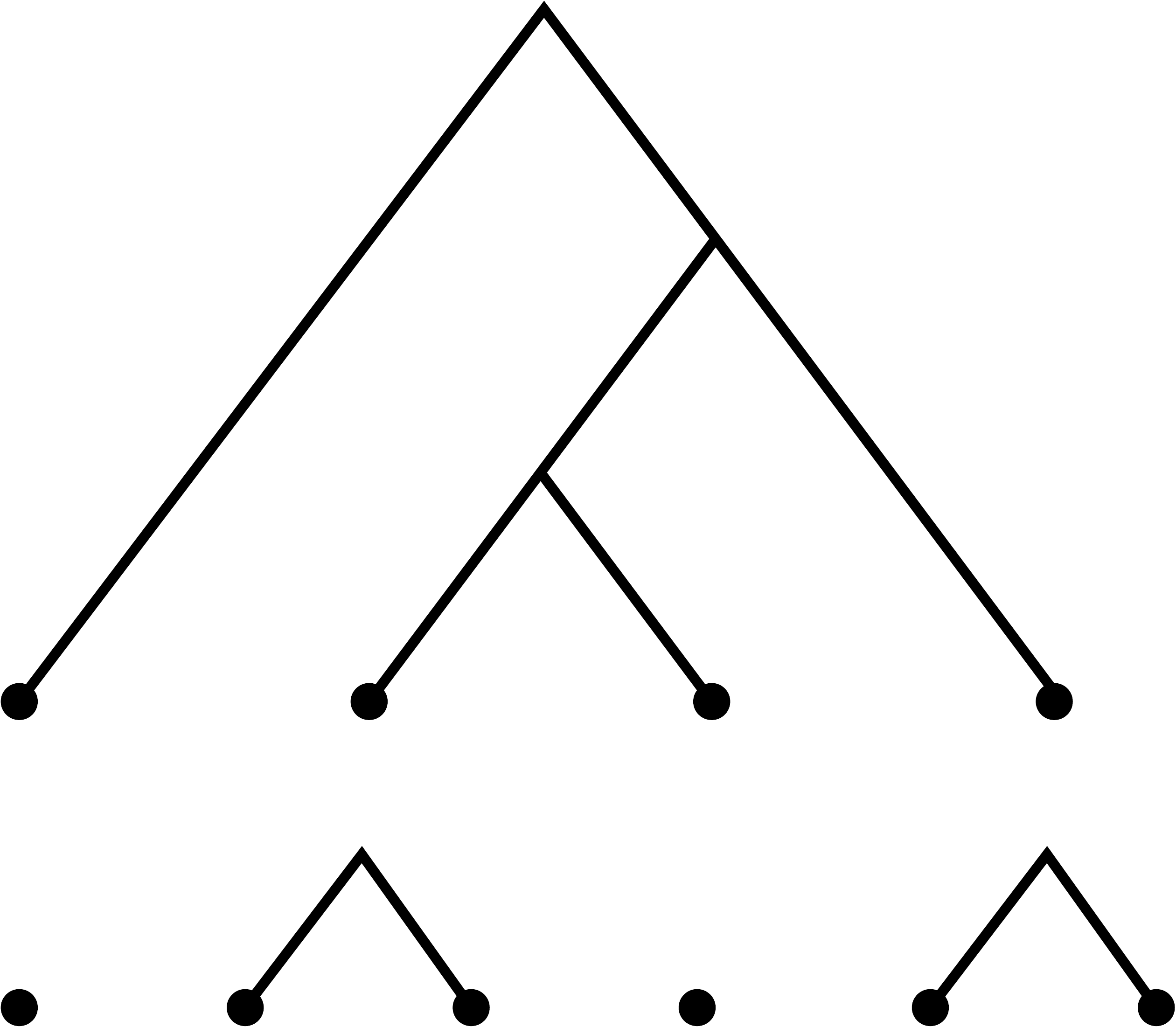}
}
\centerline {A 4-node tree $T$ of complexity 2 and a 4-tree forest $R$}

\vspace{.2in}
\centerline {
\includegraphics[width=3in]{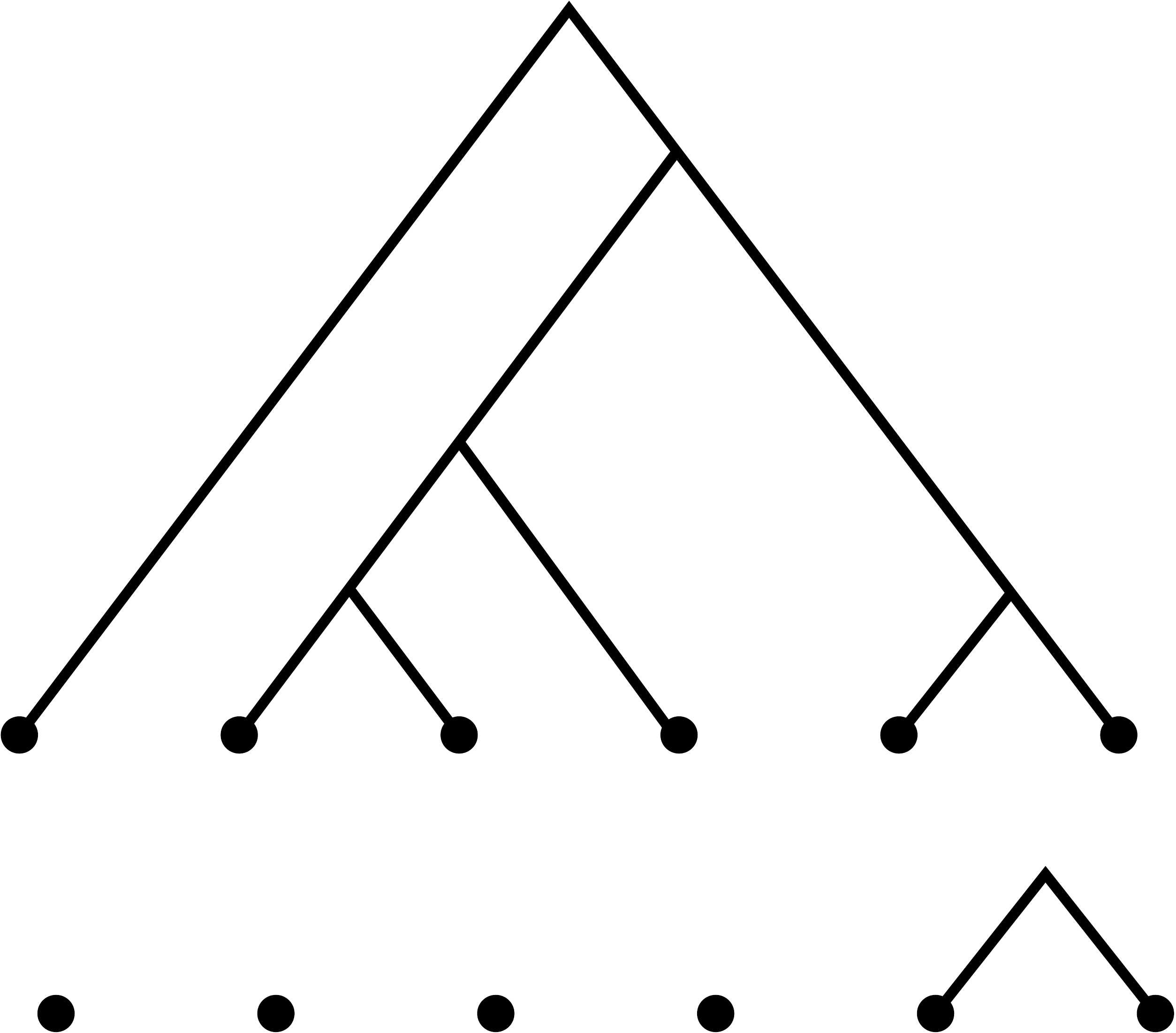}
}
\centerline {$RT$ and $s^2(RT)$}

\bigskip

\begin{proof} This is easy to see, as we can determine whether a caret is removed by $s^j$ by examining its relationship with those above it.  Namely, when we examine the unique path from a caret to the root of the tree, it consists of moves to the right and moves to the left.  An application of $s$ removes all carets whose path consists only of moves to the right. Further, any caret's path to the root hits the left edge at some point, and consists only of moves to the right afterwards.  After $s$ is applied, the path ends the move before reaching the left edge (which was a move to the left).  Thus $s$ leaves each new path from a remaining caret to the root of its new tree with one less move to the left.  So $s^j$ removes all carets whose paths consist of $j-1$ or fewer moves to the left.  Since this property is unchanged in the carets of $T$ whether or not it sits on $R$, the effect of $s^j$ is the same on carets of $T$, whether or not we consider it a top-tree of $R$.
\end{proof}

For a positive integer $j$, we define a function $\phi_j$ from pointed forests to forests in the following way:  Apply $s$ to the pointed tree and every tree to its left $j$ times .  Apply $s$ $(j-q)$ times to the tree that is $q$ to the right of the pointed tree.  That is to say, apply $s$ ($j-1$) times to the tree to the immediate right of the pointed tree, $j-2$ times to next tree to the right, etc.

\begin{proposition} \label{P:charGam}

A pointed forest $P$ lies in $\Gamma_k^l$ if and only if $\phi_l(P)$ has $k$ or fewer carets.

\end{proposition}

{\em Proof of $\Rightarrow$}:  Let $P \in \Gamma_k^l$.  First suppose that $k=0$.  In this case $P$ can be expressed as a word $v$ in $\{x_0,...,x_l\}$, and the proposition says it is annihilated by $\phi_l$, i.e., $\phi_l(P)$ consists only of trivial trees.  We proceed by induction on the length of $v$.  Clearly, $vx_0$ is annihilated by $\phi_l$ if $v$ is, since $\phi_l(vx_0)$ is a subforest of $\phi_l(v)$ (each tree has $s$ applied to it either the same number of times or one more time, since the pointer has simply moved one tree to the right).

For $0 < i \leq l$, multiplying by $x_i$ adds a caret to the right of the tree $i-1$ trees from the pointer.  By \ref{P:calcCpx}, this will increase the number of applications of $s$ required to make that tree trivial only if its right child has complexity greater than or equal to its left child.  In this case, the new tree's complexity will be one greater than that of its right child.  Since $i \leq l$, the right child was $i$ trees to the right of the caret, and by induction the right child had complexity no more than $(l-i)$.  Thus the new tree has complexity no more than $(l-i+1)$.  Since this new tree is $i-1$ trees to the right of the pointer, it is still annihilated by $\phi_l$.  The trees to the left of the new caret are unchanged, and the trees to the right of the caret have each been brought 1 tree closer to the pointer since two intervening trees have been merged.  These trees now have complexity $\leq \max(l - d - 1, 0)$, where $d$ is their distance from the pointed tree, since they were distance $d+1$ before the caret was added and were made trivial by the application of $\phi_l$.  This means that they will still be annihilated by $\phi_l$, and so the new pointed forest is still turned into a trivial forest by $\phi_l$.

The above argument shows that $\phi_l(v)$ is trivial if $v \in \Gamma_0^l$.  But if we take a $w$ of length $\leq k$ as in the theorem, $wv$ uses the trees of $w$ as the leaves for the trees of $v$.  Thus each all the carets added in each tree of $v$ are still removed by $\phi_l$ by \ref{P:skills}, since $s$ is applied the same number of times to each tree.  Thus $\phi(wv)$ has at most as many carets as $w$, i.e., $k$ or fewer. \qed

To prove the reverse direction of \ref{P:charGam}, we will use the following proposition:

\begin{proposition} \label{P:buildT}

A pointed forest diagram consisting of a single nontrivial tree $T$ of complexity $j$ in the leftmost position, with the pointer on that tree, can be expressed as as word in $x_1, ..., x_j$.

\end{proposition}

\begin{proof} We will proceed by induction on $j$.  It is clear that $s(T)$ is a forest with trivial leftmost tree and the remaining trees $T_1, ... ,T_n$ of complexity $\leq j-1$.  By induction, let $u$ be a word in $x_1, ... , x_{j-1}$ that creates the forest diagram with only $T_1$.  Now consider $x_0ux_0^{-1}$.  Since $x_0$ moves the pointer one to the right and $x_0^{-1}$ moves it one to the left, $x_0ux_0^{-1}$ represents creating $T_1$ one tree to the right of the pointer (leaving the pointed tree trivial).  But since $x_i = x_0x_{i-1}x_0^{-1}$, by inserting $x_0^{-1}x_0$ between each letter in $u$ we can rewrite $x_0ux_0^{-1}$ as a word in $x_2,...,x_j$.  Applying $x_1$ then creates the caret from the leftmost node to this tree (this would be the last caret removed by the first application of $s$ to $T$). We may now repeat this process for the remaining trees:  Multiply by some $x_0ux_0^{-1}$ to create $T_i$ one tree to the right of the pointer, and then multiply by $x_1$ to attach the caret removed by the first application of $s$ (this process is illustrated below).  This constructs the entire tree $T$.
\end{proof}

\vspace{.2in}
\centerline {
\includegraphics[width=4.3in]{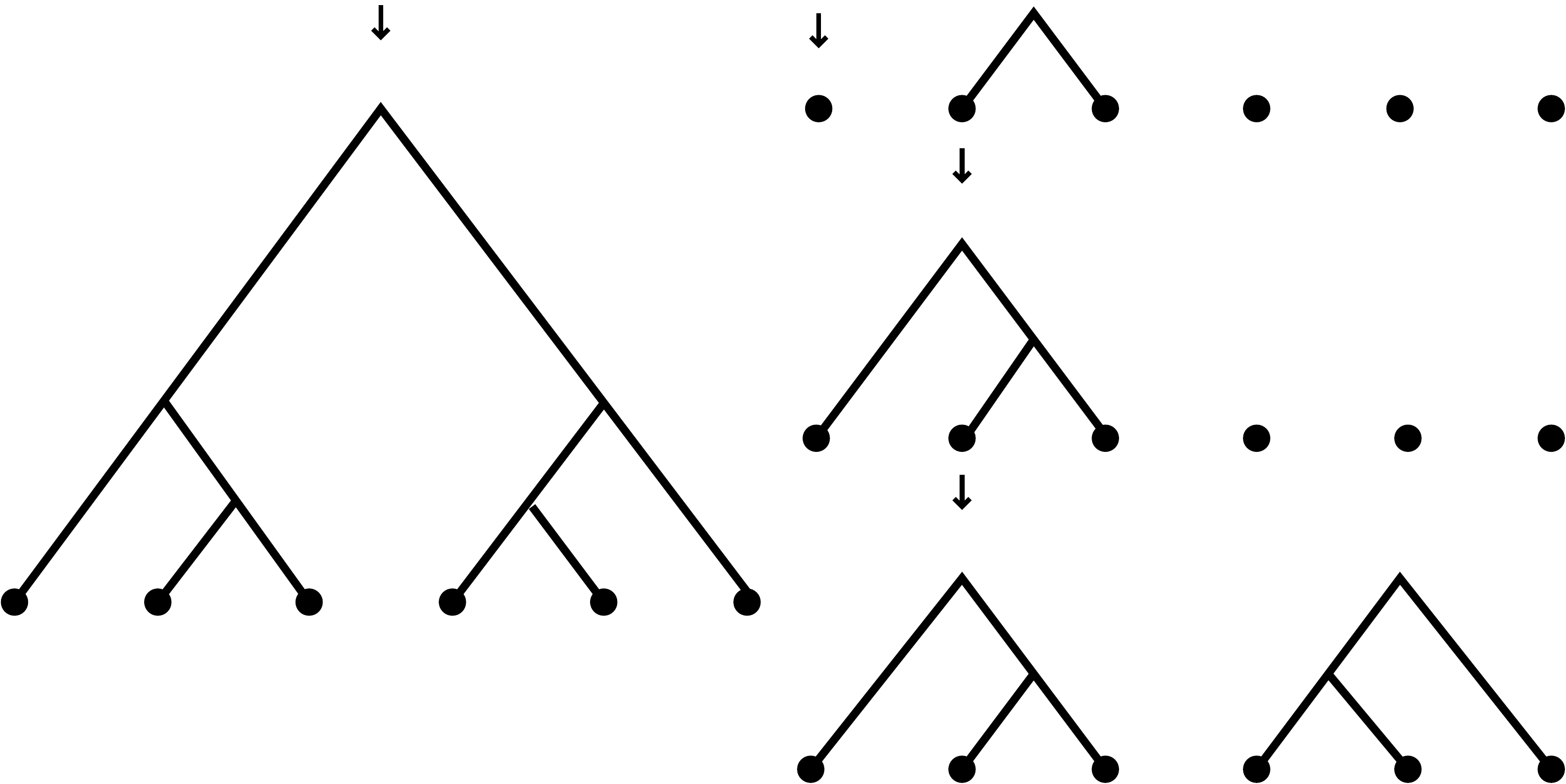}
}
{\center \small To construct the tree $T$ on the left, we construct the first nontrivial tree in $s(T)$ to the right of the pointer (top right), then multiply by $x_1$ (middle right), then construct the next tree in $s(T)$ to the right of the pointer (bottom right), then multiply by $x_1$.}

\vspace{.2in}

{\em Proof of Proposition \ref{P:charGam}, $\Leftarrow$:} Suppose that $P$ is a pointed forest such that $\phi_l(P)$ has $k$ or fewer carets. We can then create $w$ to put these carets in place without moving the pointer (the generator $x_i$ creates a caret on the $i^{th}$ tree without moving the pointer).

Consider the element $w^{-1}P$.  This is the pointed forest obtained by taking the trees in $P$ which remain after applying $\phi_l$, and replacing them with trivial trees:

\vspace{.2in}
\centerline {
\includegraphics[width=4.3in]{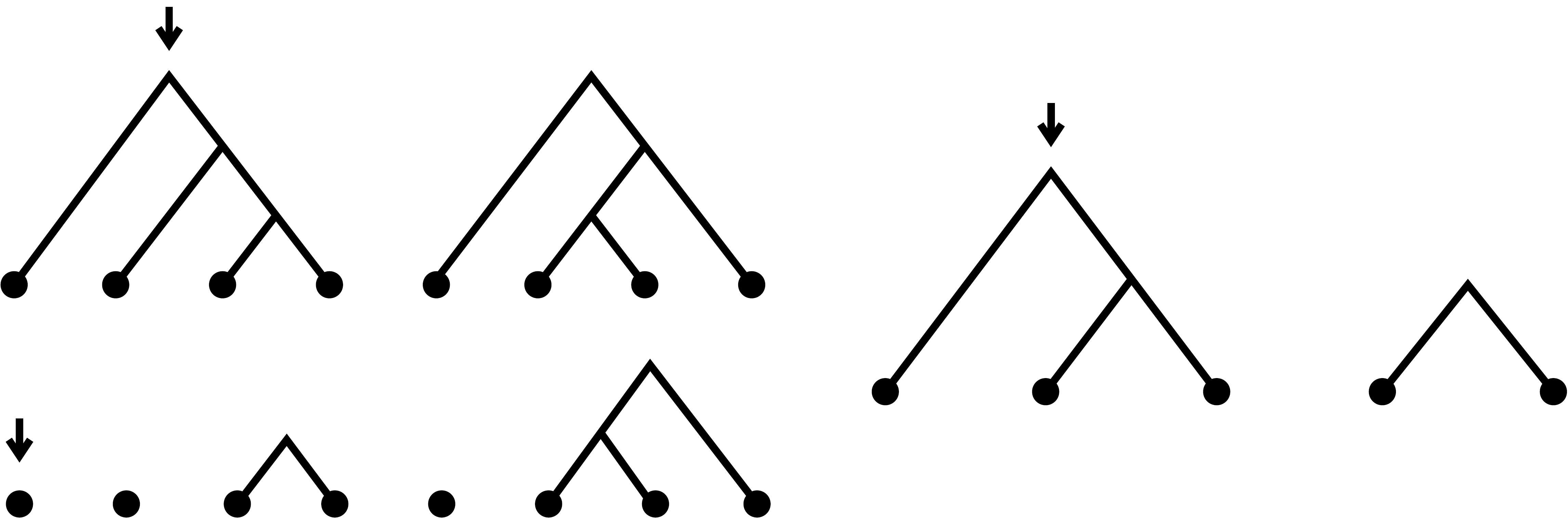}
}

{\center \small For $l = 2$, if $P$ is the forest in the top left then $w$ is $\phi_2(P)$ with the pointer on the first tree (bottom left), and $w^{-1}P$ is shown on the right.}

\bigskip

The resulting pointed forest is then annihilated by $\phi_l$, and so each tree under or to the left of the pointer has complexity at most $l$.  Thus we may construct these trees as words in $x_0,...,x_l$ using \ref{P:buildT} and inserting $x_0$ between each word, which will result in building the first tree, moving the pointer to the right, building the next tree, etc.  Further, the tree that is $j$ trees to the right of the pointer has complexity at most $l-j$, and so \ref{P:buildT} says we can construct it as $x_0^jux_0^{-j}$, where $u$ is a word in $x_1,...,x_{l-j}$.  As above, we then insert $x_0^jx_0^{-j}$ between each letter of $u$, which allows us to rewrite it as a word in $x_{j+1},...,x_l$.  Repeating this for each $j$ and appending these words in increasing order constructs all trees to the right of the pointer.  This completes the construction of $w^{-1}P$ as a word in $x_0,...,x_l$, which we will call $v$.  Thus, $P = ww^{-1}P = wv$, and the proof is complete. \qed

\bigskip

{\em Proof of Theorem \ref{T:main}:} Let $F$ be the underlying forest of a pointed forest in $\Gamma^l_k$.  Let $Q$ be the pointed forest with underlying forest $F$ whose pointer is as far to the left as possible while still remaining in $\Gamma^l_k$.  Note that applying $\phi_l$ to $Q$ affects at most $l$ trees under or to the right of the pointer.  Thus, by \ref{P:charGam} there are at most $k+l$ nontrivial trees to the right of the pointer in $P$, otherwise, $\phi_l(Q)$ would have more than $k$ nontrivial trees and thus certainly have more than $k$ carets.

For any $P \in \Gamma^l_k$, let $c_P$ be the number of nontrivial trees $T$ to the left of the pointed tree such that if the pointer is moved to $T$ the resulting forest remains in $\Gamma^l_k$.  By the preceding paragraph $c_P$ is never more than $k+l$.  Now define $G: \Gamma^l_k \times \Gamma^l_k \rightarrow \mathbb{Z}$ as follows:  For each pointed forest $P$,

\bigskip

-If $Px_0^{-1} \in \Gamma^l_k$, set $G(P, Px_0^{-1}) = c_P$ and $G(Px_0^{-1}, P) = -c_P$.

\smallskip

-If the pointed tree in $P$ is nontrivial, and $Px_1^{-1} \in \Gamma_k^l$, set $G(P, Px_1^{-1}) = 1$ and $G(Px_1^{-1}, P) = -1$.

\smallskip

-For all other pairs $(P, P')$, set $G(P, P') = 0$.

\bigskip

Now it is clear from the definition that $G$ satisfies conditions ($i$) and ($ii$) in Definition \ref{D:PonziFlow}, and since $c_P \leq k+l$ for each $P$, $G$ also satisfies condition ($iii$).  It thus remains only to check condition ($iv$).  So we shall consider a pointed forest $P \in \Gamma^l_k$.

Note that $G(Px_1, P) + G(Px_1^{-1}, P) \geq 0$.  Note also that that $G(Px_0, P) + G(Px_0^{-1}, P) \geq 0$, since any tree to the left of of the pointed tree in $P$ is also to the left of the pointed tree in $Px_0$, so $c_P \leq c_{Px_0}$.  Further, since the only tree counted in $c_{Px_0}$ but not in $c_P$ is the pointed tree in $P$, we have that $G(Px_0, P) + G(Px_0^{-1}, P) > 0$ exactly when the pointed tree in $P$ is nontrivial.  Thus condition (iv) is satisfied for all $P$ where the pointed tree is nontrivial.  If the pointed tree in $P$ is trivial, $G(Px_1, P) = 1$ and $Px_1^{-1}$ is not in $\Gamma^l_k$, so $\sum_{P'\in \Gamma^l_k} G(P',P) = G(Px_0, P) + G(Px_0^{-1}, P) + G(Px_1, P) \geq 1.$  Thus condition ($iv$) holds for all pointed forests where the pointed tree is trivial as well.  \qed

We close with some immediate corollaries:

\begin{corollary}
If $F$ is amenable, then for any right-invariant measure $\mu$, $\mu(\Gamma^l_k) = 0$.
\end{corollary}

\begin{proof}
By Theorem \ref{T:main}, $\Gamma^l_k$ has a Ponzi flow, and thus by \ref{C:MeasureZero} it always has measure zero.
\end{proof}

\begin{corollary}
If $F$ is amenable, then for any right-invariant measure $\mu$, and any finitely generated submonoid $P'$ of $P$, $\mu(P') = 0$.
\end{corollary}

\begin{proof}
Letting $p_1, ..., p_n$ be generators of $P'$, express each as a word in the $x_0, x_1, x_2, ...$ generating set.  Let $L$ be the maximum index of the $x_i$ used to express the $p_j$; then $P'$ is a subset of the monoid generated by $x_0, x_1, ..., x_L$.  But this monoid is exactly $\Gamma_0^L$, which by the previous corollary has measure zero.  Thus, $\mu(P') = 0$.
\end{proof}


\begin{thebibliography}{99}
\bibitem{jB04}
    J. Belk, {\em Thompson's Group F}, PhD Thesis, Cornell University (2004).
\bibitem{BW92}
    J. Block, S. Weinberger, {\em Aperiodic Tilings, Positive Scalar Curvature, and Amenability of Spaces}, Journal of the Amer. Math. Soc. {\bf 5} (1992), 907--918.
\bibitem{BS85}
    M. Brin and C. Squier, {\em Groups of piecewise linear homeomorphisms of the real line}, Invent. Math. {\bf 79(4)}(1985), 485--498.
\bibitem{CF96}
    J.W. Canon, W.J. Floyd, and W.R. Parry, {\em Introductory notes on Richard Thompson's Groups},  Ensign. Math. (2) {\bf 42(3-4)} (1996), 215--256.
\bibitem{FO55}
    E. F{\o}lner, {\em On groups with full Banach mean values}, Math. Scand. {\bf 3} (1955), 243--254.
\bibitem{NA64}
    I. Namioka, {\em F{\o}lner's conditions for amenable semi-groups}, Math. Scand. {\bf 15} (1964), 18--28.
\bibitem{OS02}
    A.J. Ol'shanskii and M.V. Sapir, {\em Non-amenable finitely presented torsion-by-cyclic groups}, Publ. Math. Inst. Hautes \'{E}tudes Sci. {\bf 96} (2002), 43--169.
\bibitem{DS08}
    D. Savchuk, {\em Some graphs related to Thompson's Group F}, preprint.
\bibitem{LY99}
    I. Benjamini, and R. Lyons, Y. Peres, and O. Schramm, {\em Group-invariant percolation on graphs}, Geom. Funct. Anal., {\bf 9} (1999), 29--66.

\end{thebibliography}
\end{document}